\documentclass[12pt]{amsart}
\usepackage{amsfonts}
\usepackage{latexsym,amsmath,amsthm,amssymb,amsfonts,mathrsfs}
\usepackage{graphicx}

\usepackage[dvipsnames]{xcolor}

\numberwithin{equation}{section}

\newtheorem{thm}{Theorem}[section]

\newtheorem{lem}{Lemma}[section]
\newtheorem{prop}{Proposition}[section]

\theoremstyle{definition}
\newtheorem{dfn}{Definition}[section]

\theoremstyle{remark}
\newtheorem{rem}{Remark}

\newcommand{\End}{\text{End}}

\newcommand{\xc}{e^{-\eta}\phi}

\begin{document}

\title{Uniqueness theorems for non-compact mean curvature flow with possibly unbounded curvatures}

 \author{Man-Chun Lee}
\address[Man-Chun Lee]{Department of
 Mathematics, The Chinese University of Hong Kong, Shatin, Hong Kong, China.}
\email{mclee@math.cuhk.edu.hk}

 \author{Man Shun John Ma}
\address[Man Shun John Ma]{Department of Mathematics, Rutgers, The State University of New Jersey, New Brunswick}
\email{john.ma311@rutgers.edu}

\begin{abstract}
In this paper, we discuss uniqueness and backward uniqueness for mean curvature flow of non-compact manifolds. We use an energy argument to prove two uniqueness theorems for mean curvature flow with possibly unbounded curvatures. These generalize the results in \cite{ChenYin}. Using similar method, we also obtain a uniqueness result on Ricci flows. A backward uniqueness theorem is also proved for mean curvature flow with bounded curvatures.
\end{abstract}

\date{\today}

\maketitle

\markboth{Man-Chun Lee and John Man Shun Ma}{Uniqueness on Mean Curvature Flow with unbounded curvature}

\section{Introduction}
Given an immersion $F_0: \Sigma \to M$, the mean curvature flow (MCF) starting at $F_0$ is a family of immersions which moves along the mean curvature vector. The MCF is the negative gradient flow of the area functional, and has been studied extensively for the past 40 years.

When $\Sigma$ is compact, the MCF starting at an immersion $F_0$ always exists and is unique up to a maximal time interval $[0,T)$. On the other hand, if $\Sigma$ is non-compact, the general existence and uniqueness problem is still not solved.

The first existence result in the non-compact setting is by Ecker and Huisken \cite{EckerHuisken}, where they prove the existence of MCF starting at a hypersurface $M_0$ in $\mathbb R^{n+1}$ with uniform Lipschitz bound. If $M_0$ is an entire graph, they also show the long time existence when $M_0$ is merely locally Lipschitz.

We remark that the minimal Lipschitz cone constructed by Lawson and Osserman \cite{LawsonOsserman} may serve as an obstruction to the apriori estimates in \cite{EckerHuisken} in higher codimension. Existence of non-compact MCF has only been obtained for entire graph with assumptions on smallness of Lipschitz norms \cite{ChauChenHe}, \cite{ChauChenYuan}, \cite{KochLamm}, \cite{Wang}, \cite{Lubbe}.

Next we discuss the uniqueness of MCF. Koch and Lamm show uniqueness of MCF \cite{KochLamm} for entire graph with small Lipschitz bound in any codimension. Chen and Peng prove in \cite{ChenPeng} that any viscosity solution of the graphical
Lagrangian MCF with a continuous initial data is unique. For general immersions, Chen and Yin show in \cite{ChenYin} the uniqueness of MCF among flows with uniformly bounded second fundamental forms. Together with a pseudolocality theorem, they prove uniqueness of MCF starting from an proper embedding with bounded second fundamental form and satisfying an uniform graphic condition.

The first goal of this paper is to prove the following uniqueness theorem which generalizes Chen and Yin's uniqueness result to the case of possibly unbounded curvatures.

\begin{thm} \label{Uniqueness theorem}
Let $(M,h)$ be a non-compact complete Riemannian manifold with positive injectivity radius lower bound $i_0$ such that
\begin{equation} \label{assumptions on ambient curvatures}
|\bar R|\le B_0, |\bar\nabla \bar R| \le B_1 \text{ and } |\bar\nabla^2 \bar R|^2 \le L r^{2-\epsilon} \text{ for } r>>1,
\end{equation} where $\bar\nabla, \bar R$ are respectively the Levi-Civita connection and the Riemann curvature tensor of $(M, h)$, $r(y) = d_M (y,y_0)$ for some $y_0\in M$ and $B_0, B_1, L,\epsilon>0$. Let $F_0:\Sigma \to M$ be a smooth proper immersion so that
\begin{equation} \label{assumption on volume growth of F_0}
\text{Vol}_{\Sigma} (F_0^{-1} B_M (y_0, \bar r)) \le D e^{D \bar r^2}
\end{equation}
for some constant $D>0$ and for all $\bar r>>1$. Let $F, \widetilde F$ be smooth solutions to the MCF starting at $F_0$, which satisfy the following conditions:
\begin{enumerate}
\item $F, \widetilde F$ are uniformly continuous with respect to $t$, 
\item The induced metric $g(t), \tilde g(t)$ on $\Sigma$ are uniformly equivalent to $g_0$, and
\item the second fundamental forms $A$, $\widetilde A$ satisfy
\begin{equation} \label{assumption on A}
|A|^2(t,x) + |\widetilde A|^2 (t,x) \le \frac{L}{t} r^{2-\epsilon} (F_0(x))
\end{equation}
for some $L>0$.
 \end{enumerate}
Then $F = \widetilde F$.
\end{thm}

The precise definition in condition (1) and (2) is given in section 2. 

If we compare the above theorem to theorem 1.1 and 1.3 in \cite{ChenYin}, we assume weaker curvature bounds, while in the expense of assuming the volume growth (\ref{assumption on volume growth of F_0}) of the initial immersions. We do not make any graphic/curvature assumptions on the initial immersion. We remark that condition (1) and (2) are both satisfied if $|A|, |\widetilde A|$ are uniformly bounded.

Note that in theorem \ref{Uniqueness theorem}, we assume that $F_0$ is smooth and the MCF $F$ and $\widetilde F$ are both smooth up to time $t=0$. In the next theorem, it is shown that  under a better bound on the second fundamental forms and $\bar R$, one can relax these assumptions and drop the volume growth condition.

\begin{thm} \label{Uniqueness theorem for t^alpha bound}
Let $(M,h)$ be a non-compact complete Riemannian manifold with positive injectivity radius lower bound and uniformly bound on $\bar\nabla^i \bar R$ for $i=0,1,2$. Let $F_0:\Sigma \to M$ be a $C^3$ proper immersion and let $F, \widetilde F : (0,T] \times \Sigma \to M$ be solutions to the MCF so that $F, \widetilde F$ converges to $F_0$ locally in $C^3$ as $t\to 0$. If the curvatures satisfy $|A|+|\widetilde A| \le L/t^\alpha$ for some $\alpha <1/2$ and $L>0$. Then $F = \widetilde F$.
\end{thm}

The proof of the above theorems, like all other uniqueness results in MCF, use the parabolicity of the MCF equation. The technical issue is that the equation is not strictly parabolic - it's invariant under diffeomorphisms. In the previous approaches \cite{ChenYin}, \cite{ChauChenHe}, \cite{ChauChenYuan}, \cite{KochLamm}, \cite{Wang}, \cite{Lubbe}, they use the well-known De Turck trick to construct a family of diffeomorphisms so that the resulting equation (Mean curvature De Turck flow) becomes strictly parabolic (note that the use of De Turck tricks are implicit in the graphical case, see for example p.548-549 in \cite{EckerHuisken}).

Our proof of theorem \ref{Uniqueness theorem} and theorem \ref{Uniqueness theorem for t^alpha bound} use directly the parabolic equation satisfied by the second fundamental form. We employ an energy argument first performed by Kotschwar in \cite{Kotschwar}, where he proves a uniqueness result for non-compact Ricci flow. The energy argument was then used again for other geometric flows \cite{Kotschwar2}, \cite{Kotschwar5}, \cite{LotayWei}, \cite{BHV}, \cite{Bell}, \cite{Alen}, \cite{SongWang}. The main idea is to consider the quantity
$$E(t) = \int_{\Sigma} \mathcal Q d\mu_t,$$
where $\mathcal Q$ is a quantity so that $\mathcal Q = 0$ implies uniqueness. The goal is to show $E(t)=0$ given that $E(0) = 0$. For example, in the Ricci flow situation \cite{Kotschwar}, $\mathcal Q$ contains a term of the form $t^{-\beta} |g-\tilde g|^2$ for some $\beta$, where both $g, \tilde g$ are solutions to the Ricci flow.

In our situation, we choose our $\mathcal Q$ to contain the zeroth order term
\begin{equation} \label{d in intro}
d_{M} (F(t,x), \widetilde F(t,x)).
\end{equation}
As we will see later, first and second order terms should also be present in $\mathcal Q$ in order to obtain a nice differential inequality for $E(t)$. We do not need higher order quantity though: the parabolic nature of MCF gives a nice parabolic equation for the second fundamental form $A$, and an integration by part give a strictly {\bf negative} term containing the third order quantities $\nabla A$, which cancels all other third order quantities. As we will see later, cut-off functions are inserted in the energy $E$ to deal with the non-compact situation. 

Let us point out one key technical difference between our works and those in \cite{Kotschwar}, \cite{Kotschwar2}, \cite{LotayWei}, \cite{BHV}, \cite{Bell}, \cite{Alen}: In their energy arguments, the flows they consider are intrinsic, as opposed to MCF which is extrinsic. Not only that both the curvatures of $\Sigma$ and $M$ play a role, but also that the geometric quantities of two a-priori different MCFs live in different vector bundles on $\Sigma$. Thus one needs to use a bundle isomorphisms $P$ to identity these bundles before estimating the difference. In our situation, we construct $P$ using a parallel transport along the shortest paths between two MCFs. We remark that the same construction is also carried out in \cite{SongWang}, \cite{McGahagan} in the context of Schr{\" o}dinger flow. We expect the same argument should work for other extrinsic geometric flows.

As a by-product, we obtain the following uniqueness result for Ricci flow, which generalizes results in \cite{ChenZhu}, \cite{Kotschwar}.

\begin{thm} \label{RicciFlow}
Let $(M,g_0)$ be a smooth complete noncompact Riemannian manifold. Let $g(t),\tilde g(t), t\in [0,T]$ be two smooth complete solutions to Ricci flow with initial metric $g_0$. Suppose that $g(t), \tilde g(t)$ are uniformly equivalent to $g_0$ and 
$$|Rm|+|\widetilde{Rm}|\leq \frac{L}{t}$$
for some constant $L>0$. Then $g(t)=\tilde g(t)$ for all $t\in [0,T]$.
\end{thm}

In contrast with the result in \cite{Alen}, we do not assume any growth rate on $|Rm(g_0)|_{g_0}$ and the size of $L$, while in the expense of assuming the uniform equivalence of metrics. The last main result is the following backward uniqueness theorem for MCF.

\begin{thm} \label{Backward uniqueness theorem}
Let $(M, h)$ be a complete non-compact Riemannian manifold with positive lower bound on injectivity radius and uniform upper bound on $|\bar\nabla^i \bar R|$ for $i\le 4$. Let $F, \widetilde F : [0,T]\times \Sigma \to M$ be smooth MCFs with uniformly bounded second fundamental forms. If $F(T, \cdot) = \widetilde F(T, \cdot)$, then $F = \widetilde F$ for all $t$.
\end{thm}

The proof of theorem \ref{Backward uniqueness theorem} uses again that the second fundamental form $A$ and its derivatives $\nabla A$ both satisfy strictly parabolic equations. While the lower order quantities do not, one can show that they satisfy an ordinary differential inequality. These coupled inequalities are sufficient to show theorem \ref{Backward uniqueness theorem} by a general backward uniqueness theorem in \cite{Kotschwar4}. The reader may find more historical remarks in the introduction of \cite{Kotschwar4}.

One slight technical issue is that the distance (\ref{d in intro}) is non-differentiable when it's zero and thus we need another zeroth order quantity. We treat $\widetilde F$ as a graph of $F$ and represent $\widetilde F$ by a section on the pullback bundle $F^{-1}TM$. The assumptions on the fourth covariant derivatives of $\bar R$ in theorem \ref{Backward uniqueness theorem} is used in estimating the parabolic equation for $\nabla A$, which we do not need in the proof of theorem \ref{Uniqueness theorem}.

{When the ambient space $M$ in Theorem \ref{Backward uniqueness theorem} is Euclidean, the result is proved in \cite{Huang} in the co-dimension one case and recently in \cite{Zhang} for arbitrary co-dimensions.}

In section 2, we fix the notations and prove some elementary results. The parallel transport $P$ will be studied in section 3 and 4. The main estimates are performed in section 5. Theorem \ref{Uniqueness theorem}, theorem \ref{Uniqueness theorem for t^alpha bound} and theorem \ref{RicciFlow} are proved in section 6 and theorem \ref{Backward uniqueness theorem} is proved in section 7.

{\bf Acknowledgement} The first author would like to thank Professor Luen-Fai Tam for his constant support and encouragement. The second author would like to thank Professor Jingyi Chen for the discussions and pointing out the reference \cite{SongWang}. Part of the works was done when the second author visited The Chinese University of Hong Kong and he would like to thank Professor Martin Li for the hospitality.


\section{Prelminary and notations}
In this section, we review some definitions and results in basic submanifold theory and MCF. Let $\Sigma$ be a smooth manifold and $(M, h)$ be a smooth Riemannian manifold. Let
$$F, \widetilde F : \Sigma \times [0,T] \to M$$
be two families of smooth immersions.

Next we introduce several notations. We write only the notations for $F$. A tilde will be added to the corresponding notations for $\widetilde F$. We use $(x^1,\cdots x^n)$ and $(y^1, \cdots, y^N)$ respectively to denote the local coordinates on $\Sigma$, $M$. We use $i,j,\cdots$ to denote the indices of $\Sigma$, $\alpha, \beta, \cdots$, and $\alpha', \beta' \cdots$ respectively to denote the indices on $F^{-1}  TM$, $\widetilde F^{-1}TM$. For each $t$, let $g=g(t)= F(t,\cdot)^*h$ and $\nabla $ be the Levi-Civita connection with respect to $g$. We use the same notation $\nabla$ to denote the induced connection on all $(p,q)$-tensor bundle $T^{p,q}\Sigma$.

We say that the family of immersions $F$ is uniformly continuous with respect to $t$, if for all $\delta >0$ and $s\in [0,T)$, there is $s_\delta >0$ so that
\begin{align} \label{F C^0  close to F_0}
d_M (F(t, x) ,F(s,x))\le \delta, \ \ \ \forall (t,x)\in [s,s+s_\delta]\times \Sigma.
\end{align}

We say that the induced metric $g(t)$ is uniformly equivalent to $g_0 = F_0^*h$, if there is $\lambda >1$ so that
\begin{equation} \label{uniformly equivalence metric}
\lambda^{-1}g_0\leq g(t) \leq \lambda g_0, \ \ \ \forall (t,x)\in [0,T]\times \Sigma.
\end{equation}

On the pullback bundle $N:=F^{-1} T M$ we have the connection induced from $h,F$:
\begin{equation} \label{connection on F-1 TM}
\nabla_i^F Y^\alpha = \partial_i Y^\alpha+ \Gamma^\alpha_{\beta \gamma}F^\beta_i  Y^\gamma, \ \ \ Y\in \Gamma( M, N).
\end{equation}
Here $\Gamma^\alpha_{\beta\gamma}$ denote the Christoffel symbols of the Levi-Civita connection $\bar\nabla$ on $(M,h)$. Note that $\Gamma^\alpha_{\beta\gamma}$ is indeed $\Gamma^\alpha_{\beta\gamma}\circ F$, but we suppress $F$ for simplicity. We also remark that
$$ \nabla_i Y = \bar\nabla_{F_i} \widetilde Y,$$
where $F_i = \partial_i F$ and $\widetilde Y$ is any extension of $Y$ in $M$. We use the same notation $\nabla ^F$ to denote the connection induced by $\nabla $ and $\nabla ^F$ on any $N$-valued tensor bundle. Thus there could be six notations in total:
\[ \nabla , \widetilde\nabla, \nabla ^F, \nabla ^{\widetilde F}, \widetilde \nabla ^F, \widetilde \nabla^{\widetilde F}.\]
However, for simplicity we use only $\nabla$ and $\widetilde \nabla$. It will be clear from the context which connection we are using.

Next we consider covariant derivatives with respect to time. Define the covariant time derivative $D_t$ on $\Gamma(\Sigma, T^{0,p}\Sigma \otimes  N)$ by
\begin{equation*}
D _t Y^\alpha_{i_1\cdots i_p} = \partial_t Y^\alpha_{i_1\cdots i_p} + \Gamma^\alpha_{\beta\gamma} Y^\beta_{i_1\cdots i_p} F^\gamma_t.
\end{equation*}
Note that when acts on vector fields along $F$, $D_t$ is metric with respect to $h$. That is,
\begin{equation*}
\partial_t h (Y, Z) = h(D_t Y, Z) + h(Y, D_t Z),\ \ \ Y, Z \in \Gamma(\Sigma, N).
\end{equation*}

Next we introduce several standard geometric quantities from an immersion. For each fixed $t$, the differential of $F(t,\cdot) : \Sigma \to M$ is denoted $F_*$, thus 
\begin{equation*}
(F_*)_i^\alpha = F^\alpha_i =\frac{\partial F^\alpha}{\partial x^i} \in \Gamma (\Sigma, T^*\Sigma \otimes N).
\end{equation*}
The second fundamental form $A$ is the covariant derivative of $F_*$:
\begin{equation*}
\begin{split}
A^\alpha_{ij} &= (\nabla F_*)^\alpha_{ij} \\
&= F^\alpha_{ij} - \Gamma^k_{ij} F^\alpha_k + \Gamma^\alpha_{\beta\delta } F^\beta_i F^\delta_j \in \Gamma(\Sigma, T^{0,2}\Sigma \otimes N) \\
\end{split}
\end{equation*}
The mean curvature vector $H$ is the trace of $A$ given by 
$$H^\alpha = g^{ij} A^\alpha_{ij} \in \Gamma (\Sigma, N).$$

We say that a family of immersions $F: [0,T] \times \Sigma\to M$ is a MCF starting from $F_0$ if
\begin{equation} \label{MCF equation}
\frac{\partial F}{\partial t} = H, \ \ \ F(0,\cdot) = F_0(\cdot).
\end{equation}

\begin{rem}
In this paper we use the following convention: We use $B_k$ (resp. $B_k^{loc}$) to denote the (resp. local) bound on $|\bar\nabla^k \bar R|$. Unless otherwise specified, we use $C$ to denote constants that depend only on the dimensions of $\Sigma$ and $M$, constants $\lambda$ in (\ref{uniformly equivalence metric}), $L$ in the statement of theorem \ref{Uniqueness theorem} and the lower bound on injectivity radius $i_0$. Constants that depend also on $B_0, B_1, \cdots, B_k$ (resp. $B^{loc}_0, B^{loc}_1, \cdots, B^{loc}_k$) are denoted $C_k$ (resp. $C_k^{loc}$). The explicit values of the constants $C, C_k, C_k^{loc}$ are not important and might change from line to line unless otherwise specified.
\end{rem}

The following simple lemma is used a lot in this paper. 

\begin{lem} \label{lemma DF bounded}
We have $|F_*|$, $|\widetilde F_*| \le C$.
\end{lem}

\begin{proof}
Note that $|F_*|^2 = h_{\alpha\beta} g^{ij} F^\alpha_i F^\beta_j = g^{ij}g_{ij} = n$. Thus $|\widetilde F_*|^2 \le C$ by (\ref{uniformly equivalence metric}).
\end{proof}

Next we recall the differential equations of the following quantities along the MCF, the proof can be found in \cite{S}.

\begin{lem} \label{lemma de of some quantities}
Under the MCF, we have
\begin{align}
\label{de of g} \partial_t g_{ij} &=-2 g^{pq}h (A_{pq}, A_{ij}), \\
\label{de of g-1} \partial_t g^{ij} &= 2g^{pq} g^{ik} g^{jl} h(A_{pq}, A_{kl}), \\
\label{de of dmu} \partial_t d\mu &= -|H|^2 d\mu, \\
\label{de of Gamma ijk} \partial _t \Gamma_{ij}^k&=-g^{kl}\left( \nabla_ih( H,A_{jl})+ \nabla_jh( H,A_{il})- \nabla_lh( H,A_{ij})  \right).
\end{align}
\end{lem}
We also need the equation for the higher covariant derivatives of $A$. Recall that in the notation of Chen and Yin \cite{ChenYin}, we have $F_* = \nabla F, A = \nabla^2 F, \nabla A = \nabla^3 F$ and so on. Proposition 2.3 in \cite{ChenYin} together with Gauss equation give

\begin{equation} \label{de of nabla^i A}
(D_t - \Delta) \nabla^k F = \sum_{l=0}^{k-1} \nabla^l [ h(A,A) * g^{-2} + \bar R * (\nabla F)^a *g^{-b}] * \nabla^{k-l} F,
\end{equation}
where $a=2$ or $4$, $b=1$ or $2$, $*$ are any contraction of tensors and $\bar R$ include any contraction of the Riemann curvature tensor on $M$ with $h^{-1}$.

In \cite{ChenYin}, the authors derive an apriori estimates for $|\nabla^k F|^2$ assuming that the second fundamental form $\nabla^2 F$ is uniformly bounded. When the second fundamental form $|A|^2(t, x)$ is bounded by {$Lr^{2-\epsilon}(F_0(x))/t$}, one can modify the proof of theorem 3.2 in \cite{ChenYin} to obtain the following lemma.

\begin{lem} \label{lemma first order estimates}
Let $L,\epsilon>0$ and let $(M, h)$ be a complete Riemannian manifold with 
$$|\bar R|\le B_0, |\bar\nabla \bar R|\le B_1, |\bar\nabla^2 \bar R|^2 \le Lr^{2-\epsilon}, $$
{where $r(y) = d_M(y, y_0)$}. Let $F$ be a MCF so that {$|A|^2(t, x) \le Lr^{2-\epsilon} (F_0(x)) /t$}. Then
\begin{equation} \label{first order estimates}
|\nabla A (t, x)| \le \frac{C_1 L r^{2-\epsilon}(F_0(x))}{t}.
\end{equation}
\end{lem}

\begin{proof}[Sketch of proof]
From proposition 3.1 in \cite{ChenYin}, we have
\begin{align*}
(\partial_t - \Delta) |\nabla^2 F|^2 &\le - |\nabla^3 F|^2 + C_1 |\nabla ^2 F|^4+ C_1, \\
(\partial_t - \Delta) |\nabla^3 F|^2 &\le - |\nabla^4 F|^2 + C(C_1 + |\nabla^2 F|^2 + |\bar\nabla^2 \bar R|^2) |\nabla ^3 F|^2 \\
&\quad + C|\nabla^3 F|^3 + C_0 |\nabla^2 F|^2 + C_0,
\end{align*}
Write $G^2 = Lr^{2-\epsilon}$. Using the conditions on $|\bar\nabla^2 \bar R|$ and $|A|$, we have
\begin{align*}
(\partial_t - \Delta) |\nabla^2 F|^2 &\le - |\nabla^3 F|^2 + C_1G^4 t^{-2} , \\
(\partial_t - \Delta) |\nabla^3 F|^2 &\le - |\nabla^4 F|^2 + C_1 |\nabla ^3 F|^3 + C_1 G^6t^{-3}.
\end{align*}
Thus one can proceed as in the proof of theorem 3.2 in \cite{ChenYin} to conclude.
\end{proof}






\section{Writing $\widetilde F$ as a graph of $F$: Basic estimates}
In this section, we represent $\widetilde F$ as a graph of $F$ and provide some basic estimates. Let
$$F, \widetilde F : [0,T] \times \Sigma \to M$$
be two families of immersions so that $F(0,\cdot) = \widetilde F(0, \cdot)$. Let $d: [0,T] \times \Sigma\to \mathbb R$ be the pointwise distance between $F$ and $\widetilde F$. That is,
\begin{equation*}
d = d(t, x) :=d_{M} (F(t, x), \widetilde F(t, x)).
\end{equation*}
For each $(t,x)$, write $p= F(t,x)$ and $\tilde p = \widetilde F(t, x)$. We assume that
\begin{equation} \label{assumption on d}
d <\min\left\{ i_0, 1, \frac{1}{\sqrt{2B^{loc}_0}}\right\},\ \ \ \forall (t, x) \in [0,T]\times \Sigma.
\end{equation}
Since $d<i_0$, for all $(t, x)$ there is a unique shortest geodesic joining $p$ and $\tilde p$. Write $\exp_p (sv)$, where $s\in [0,1]$, to denote this geodesic. Then $\tilde p = \exp_p v$, $|v| = d$ and $v$ is a smooth section on $N$. The collection of these geodesics forms a smooth homotopy
\begin{equation*}
\gamma : [0,T] \times \Sigma \times [0,1] \to M, \ \ \ \gamma (t, x, s) = \exp_p (sv)
\end{equation*}
{connecting $F$ and $\widetilde F$}. We use $\dot\gamma, J_t$ and $J_i$ to denote the derivative of $\gamma$ with respect to $s, t,$ and the coordinate $x^i$ respectively (the notations are so chosen since $J_t, J_i$ are Jacobi fields). Note also that $v, F_t$ and $F_i$ are the restrictions of $\dot\gamma, J_t, J_i$ to $s=0$ respectively.

Let $P : T_{\tilde p} M \to T_pM $ be the parallel transport along the geodesic $-\gamma$. The inverse $P^{-1}$ is the parallel transport along $\gamma$. 

On the endomorphism bundle $\text{End} (\widetilde N ,N)=\widetilde N^* \otimes N$ over $\Sigma$, {there is a connection induced from $\nabla^F$ and $\nabla^{\widetilde F}$. Together with the two connections $\nabla$ and $\widetilde \nabla$ defined on $T^{p,q}\Sigma$ by $g$ and $\tilde g$ respectively, there are two connections defined on any endomorphism valued $(p,q)$-tensor bundle $T^{p,q}\Sigma \otimes \text{End} (\widetilde N, N)$, which again we denote by $\nabla$ and $\widetilde \nabla$}. Note that the connections satisfy the Leibniz rule:
\begin{equation*}
\nabla (P Z) = \nabla P \cdot Z + P (\nabla Z) , \ \ \ Z \in \Gamma(\Sigma , T^{p,q} \Sigma \otimes \End (\widetilde N,N)).
\end{equation*}

In the following, we will derive estimates for $v$ and $d$. Since the calculations might be useful for other geometric situations, we do not assume that $F, \widetilde F$ satisfy the MCF equation except for Theorem \ref{D_t^n v = 0 at t=0} in this section. We remark that all of the estimates follow from the Jacobi field equation (and its higher order derivatives).

First we prove a useful lemma.

\begin{lem} \label{Lemma bounding Jacobi fields}
Let $p=p (\tau), \tilde p=\tilde p(\tau)$ be two curves in $M$ so that
$$0<d=d_M (p,\tilde p )< \min\{ i_0, 1/\sqrt{2B^{loc}_0}\}$$
for all $\tau$. Let $\gamma_\tau(s) = \exp_{p(\tau)} (sv_\tau)$, where $s\in [0,1]$, be the unique geodesic joining $p(\tau)$ to $\tilde p(\tau)$. Let $J$ be the Jacobi field given by the variation of geodesics $\gamma_\tau$. Then
\begin{align}
\label{bounding J} |J| &\le |\partial_{\tau }p| + C^{loc}_0 |\partial_{\tau} p|  d^2  +  2 | P\partial _{\tau} \tilde  p- \partial_{\tau} p|, \\
\label{bounding DJ} |\bar\nabla_{\dot\gamma} J| &\le C^{loc}_0 |\partial_{\tau} p| d^2  +  2| P\partial _{\tau} \tilde  p- \partial_{\tau} p| .
\end{align}
\end{lem}

\begin{proof}
The first inequality (\ref{bounding J}) follows from the second one using $|J| \le |\partial_\tau p| + \sup|\nabla_{\dot\gamma} J|$. To show (\ref{bounding DJ}), note that $J$ satisfies the Jacobi field equation
\begin{equation} \label{Jacobi field equation}
\bar\nabla_{\dot\gamma} \bar\nabla_{\dot\gamma} J + \bar R(\dot\gamma, J)\dot \gamma = 0.
\end{equation}
Let $\{ e_1, \cdots, e_n\}$ be a parallel orthonormal fields along $\gamma$. Write $J(s) = \sum J_i (s) e_i(s)$, then $\bar \nabla_{\dot\gamma} J = \sum J_i'(s) e_i(s)$. Also
\begin{equation}
\partial_\tau p = \sum J_i(0) e_i(0) \text{ and }P^{-1} \partial_\tau p = \sum J_i(0) e_i(1).
\end{equation}
By the mean value theorem, there is $\zeta \in [0,1]$ so that
$$| P\partial_\tau \tilde p - \partial_\tau p| = |\partial_\tau \tilde p - P^{-1} \partial_\tau p| = |\bar\nabla_{\dot\gamma} J (\zeta)|.$$
Thus by mean value theorem again, for any $s\in [0,1]$, there is $\zeta_s$ between $\zeta$ and $s$ so that
\begin{equation} \label{mean value theorem twice}
\begin{split}
|\bar\nabla_{\dot\gamma} J(s) | &\le\left| \sum (J_i'(s) - J_i'(\zeta)) e_i(s)\right| + |\bar\nabla_{\dot\gamma} J(\zeta)| \\
&\le \left|\sum J_i^{''} (\zeta_s) e_i(s)\right| +  | P\partial_\tau \tilde p - \partial_\tau p| \\
&= \left|\sum J_i^{''} (\zeta_s) e_i(\zeta_s)\right|+ | P\partial_\tau \tilde p - \partial_\tau p| \\
&= |\bar\nabla_{\dot\gamma}\bar\nabla_{\dot\gamma} J(\zeta_s)| + | P\partial_\tau \tilde p - \partial_\tau p| \\
&\le B^{loc}_0 d^2 |J(\zeta_s)| + | P\partial_\tau \tilde p - \partial_\tau p|\\
&\le B^{loc}_0 d^2 (|\partial_\tau p| + \sup |\bar\nabla_{\dot\gamma} J|) +| P\partial_\tau \tilde p - \partial_\tau p|,
\end{split}
\end{equation}
where we have used (\ref{Jacobi field equation}) and $|\bar R|\le B_0^{loc}$. Since $B^{loc}_0d^2 <1/2$ by assumption, (\ref{bounding DJ}) is shown with $C^{loc}_0 = 2B^{loc}_0$.
\end{proof}

\begin{rem}
Note that in lemma \ref{Lemma bounding Jacobi fields} we assume that $d>0$. Indeed the Jacobi field is not even defined at points where $d=0$. Thus in the MCF setting, the lemma cannot be applied directly when e.g. $t=0$. To deal with this, we first consider the case $d>0$ and conclude by letting $d\to 0$ (See the proof of proposition \ref{proposition estimate D_t v, Dv} for an example). 
\end{rem}

The first application of lemma \ref{Lemma bounding Jacobi fields} is the following proposition.

\begin{prop} \label{proposition estimate D_t v, Dv}
Let $F, \widetilde F$ be two families of immersions so that (\ref{assumption on d}) holds. Then we have
\begin{align}
\label{bound D_t v} |D_t v| &\le C^{loc}_0 |F_t|d^2 + 2|P\widetilde F_t - F_t|, \\
\label{bound Dv} |\bar\nabla v| &\le  C^{loc}_0 d^2 + 2|P\widetilde F_* - F_*|.
\end{align}
\end{prop}

\begin{proof}
Note that when $d> 0$, we have
\begin{align*}
D_t v &= \bar\nabla_{F_t} \dot\gamma \big|_{s=0} = \bar\nabla_{\dot\gamma} J_t \big|_{s=0},\\
\bar\nabla_i v &= \bar\nabla_{F_i} \dot\gamma \big|_{s=0} = \bar\nabla_{\dot\gamma} J_i \big|_{s=0}
\end{align*}
Thus (\ref{bound D_t v}), (\ref{bound Dv}) follows from lemma \ref{Lemma bounding Jacobi fields} when $d>0$. Assume $d=0$ at some $(t,x)$. Let $(t_i,x_i)$ be a sequence so that $(t_i, x_i)\to (t,x)$ and $d(t_i,x_i) >0$. Since $D_tv$, $\nabla v$ are smooth, the two inequalities can be shown by taking $i\to \infty$. If such a sequence does not exists, then $v$ is identically zero in a space time neighbourhood and so $D_tv = \nabla v = 0$.
\end{proof}

Since $|v| = d$, proposition \ref{proposition estimate D_t v, Dv} gives an estimate for $\partial_t d^2$. We also use the following lemma which can also be proved easily (See Lemma 2.2 in \cite{SongWang}).

\begin{lem} \label{lemma derivative of d}
We have
\begin{align}
\label{derivative of d} \partial _t d^2  &= 2h ( P \widetilde F_t - F_t,v),\\
\label{D of d} |\nabla d| &\le |P\widetilde F_* - F_*|.
\end{align}
\end{lem}


Next we need the following generalization of lemma \ref{Lemma bounding Jacobi fields}. First we need a definition.

\begin{dfn}
We say that a polynomial $Q$ of $k$-variables is universal if it has non-negative coefficients, $Q(0)=0$ and the coefficients depend only on the dimensions of $\Sigma$ and $M$. 
\end{dfn}

\begin{prop} \label{proposition estimate D_v D_J^k J}
Assume the same notations as in lemma \ref{Lemma bounding Jacobi fields}. Then for any $k=0,1,2,\cdots$,
\begin{equation} \label{estimate D_v D_J^k J}
\begin{split}
|\bar\nabla_{\dot\gamma} \bar\nabla^k_{J} J| &\le C^{loc}_{k}d^2 P^1_k\big(|\partial_\tau p|, |D_\tau \partial_\tau p|, \cdots , |D_\tau^{k} \partial_\tau p|\big) \\
&\quad + Q^1_k\big( |P\partial_\tau \tilde p - \partial_\tau p|, |P D_\tau\partial_\tau \tilde p - D_\tau \partial_\tau p|, \cdots, |P D_\tau^{k} \partial_\tau \tilde p - D_\tau^{k} \partial_\tau p|\big),
\end{split}
\end{equation}
where $P^1_k, Q^1_k$ are universal polynomials in $(k+1)$-variables.
\end{prop}

\begin{proof}
We argue by induction. The case $k=1$ is shown using (\ref{bounding DJ}) with
$$P^1_1(x_0) = x_0, \ \ \ Q^1_1(x_0) = 2x_0.$$

Assume that (\ref{estimate D_v D_J^k J}) holds for all integers strictly smaller than $k$ . First we see that
\begin{align} \label{higher order Jacobi field equation}
\bar\nabla_{\dot\gamma} \bar\nabla_{\dot\gamma} \nabla_J^k J= - \bar R(\dot\gamma, \bar\nabla_J^k J) \dot\gamma  +\sum (\bar\nabla^{i_r} \bar R)^{i_r'} * (\dot\gamma)^j * (\bar\nabla_J^{k_p} J)^{k_p'} * (\bar\nabla_{\dot\gamma} \bar\nabla_J^{l_q}J)^{l_q'}  ,
\end{align}
with $i_r\le k$ for all $r$ and
$$j + \sum_q l_q' = 2, \ \ \ \sum_p (k_p +1)k_p' + \sum_q (l_q+1)l_q' = k.$$
 That is, each term has exactly two $\dot\gamma$'s and $k$- $J$'s. When $k=0$, (\ref{higher order Jacobi field equation}) reduces to the Jacobi field equation (\ref{Jacobi field equation}) and the conditions on indices are satisfied trivially. In general, (\ref{higher order Jacobi field equation}) can be proved by induction, using again the following consequences of commuting covariant derivatives (note that $[J, \dot\gamma]=0$):
\begin{equation*}
\begin{split}
 \bar\nabla_J \bar\nabla_{\dot\gamma} \bar\nabla_{\dot\gamma} \nabla_J^k J &= \bar\nabla_{\dot\gamma}\bar\nabla_{\dot\gamma} \bar\nabla_J^{k+1} J + \bar\nabla_{\dot\gamma} (\bar R(\dot\gamma, J) \bar\nabla_J^k J) + \bar R(\dot\gamma, J)\bar \nabla_{\dot\gamma} \bar\nabla_J^k J, \\
 \bar\nabla_J \bar \nabla_{\dot\gamma} \bar\nabla_J^{l_q} J &= \bar\nabla_{\dot\gamma} \bar\nabla_J^{l_q +1} J + \bar R(\dot\gamma, J)\bar \nabla_J^{l_q}J.
 \end{split}
 \end{equation*}
Using (\ref{higher order Jacobi field equation}), $|\dot\gamma| = d$ and $|\bar\nabla_J^{k_p} J| \le |D_\tau^{k_p} \partial_\tau p| + \sup |\bar\nabla_{\dot\gamma} \bar\nabla_J^{k_p}J|$, we have
\begin{align*}
|\bar\nabla_{\dot\gamma}\bar\nabla_{\dot\gamma} \bar\nabla_J^k J| &\le C^{loc}_{k}\sum d^j (|D_\tau^{k_p} \partial_\tau p| + \sup |\bar\nabla_{\dot\gamma} \bar\nabla_J^{k_p}J|)^{k'_p} |\bar\nabla_{\dot\gamma} \bar\nabla_J^{l_q} J|^{l_q'}\\
&\quad + B_0d^2 |\bar\nabla_J^k J|,
\end{align*}
Using the induction hypothesis (note $k_q, l_q \le k-1$) and Cauchy Schwarz inequality, one easily obtains
\begin{equation} \label{induction step involving hat P and hat Q}
\begin{split}
|\bar\nabla_{\dot\gamma}\bar\nabla_{\dot\gamma} \bar\nabla_J^k J| &\le C^{loc}_{k} d^2 \widehat P_{k-1}(|\partial_\tau p| ,\cdots , |D_\tau^{k-1} \partial_\tau p|) \\
&\quad + \widehat Q_{k-1}(|P\partial_\tau \tilde p - \partial_\tau p|, \cdots ,| PD_\tau^{k-1} \partial_\tau \tilde p - D_\tau^{k-1} \partial_\tau p|) + B_0d^2 |\bar\nabla_J^k J|,
\end{split}
\end{equation}
where $\widehat P_{k-1}, \widehat Q_{k-1}$ are universal polynomials in $k$-variables. Now, as in the proof of Lemma \ref{Lemma bounding Jacobi fields}, we use mean value theorem twice to obtain
$$|\bar\nabla_{\dot\gamma} \bar\nabla_J^k J| \le \sup |\bar\nabla_{\dot\gamma}\bar\nabla_{\dot\gamma} \bar\nabla_J^k J| + |P D_\tau^k \partial_\tau \tilde p - D_\tau^k \partial_\tau p|$$
Finally we use (\ref{induction step involving hat P and hat Q}), $|\bar\nabla_J^k J| \le |D_\tau ^k \partial_\tau p| + \sup|\bar \nabla_{\dot\gamma}\bar \nabla_J^k J |$ and $B^{loc}_0d^2 <1/2$ to finish the induction step, with
\begin{align*}
P^1_k(x_0, \cdots, x_k) &= \widehat P_{k-1}(x_0, \cdots, x_{k-1}) + x_k, \\
Q^1_k(x_0, \cdots, x_k) &= \widehat Q_{k-1}(x_0, \cdots, x_{k-1}) + 2 x_k.
\end{align*}
This finishes the proof of the proposition.
\end{proof}

\begin{prop}  \label{proposition bounding nabla_J^k dot gamma}
Assume the same notations as in lemma \ref{Lemma bounding Jacobi fields}. Then
\begin{align} \label{bounding nabla_J^k dot gamma}
|\bar\nabla_{J}^k \dot\gamma| &\le C^{loc}_{k-1} d  P^2_{k-1} (|\partial_\tau p|, |D_\tau \partial_\tau p| , \cdots,|D_\tau^{k-1} \partial_\tau p|)  \\
&\quad +  Q^2_{k-1}(| P\partial_\tau \tilde p - \partial_\tau p|, |PD_\tau \partial_\tau \tilde p - D_\tau \partial_\tau p| , \cdots, |PD_\tau^{k-1}\partial_\tau \tilde p - D_\tau^{k-1} \partial_\tau p|), \nonumber
\end{align}
where $P^2_{k-1}, Q^2_{k-1}$ are universal polynomials of $k$-variables.
\end{prop}

\begin{proof}
By commuting covariant derivatives, we have
\begin{align*}
\bar\nabla_{J}^k \dot\gamma &= \bar\nabla_{\dot\gamma} \bar\nabla^{k-1}_{J} J +\sum_{i=0}^{k-2} \bar\nabla_J^ i (\bar R (\dot\gamma, J) \bar\nabla_{J}^{k-2-i} J)\\
&= \bar\nabla_{\dot\gamma} \bar\nabla^{k-1}_{J} J + \sum (\bar\nabla^{i_r} \bar R)^{i_r'} * (\dot\gamma)^j * (\bar \nabla_{J}^{k_p} J)^{k_p'} * \big(\bar\nabla_{\dot\gamma} \bar\nabla _{J}^{l} {J}\big)^{l'},
\end{align*}
where $l , k_p\le k-3$, $i_r \le k-2$ for all $r$,
$$ j + l'= 1 \text{ and } \sum_p (k_p+1) k_p' +  (l +1) l' = k.$$
By proposition \ref{proposition estimate D_v D_J^k J} and $|\bar \nabla_J^{k_p} J| \le |D_\tau^{k_p} \partial_\tau p| + \sup |\bar \nabla_{\dot\gamma} \bar\nabla_J^{k_p} J|$, one obtains (\ref{bounding nabla_J^k dot gamma}).
\end{proof}

An immediate consequence is the following theorem, which says that if $F$ and $\widetilde F$ are both MCF starting at $F_0$, then they agree infinitesimally.

\begin{thm}\label{D_t^n v = 0 at t=0}
Let $F, \widetilde F$ be smooth MCFs starting at $F_0$. Then $D_t^k v |_{t=0} = 0$ for all $k=0,1,2,\cdots$.
\end{thm}

\begin{proof}
{We use an approximation argument as in the proof of Proposition \ref{proposition estimate D_t v, Dv}. Let $x\in \Sigma$. If $v$ is identically zero in a space time neighbourhood of $(0, x)$, then $D^k_t v (x)=0$ is clear. If not, then there are $(\tau_i, x_i)\to (0,x)$ so that $v(\tau_i, x_i)\neq 0$. For each $(\tau_i, x_i)$, }
In general, whenever $v\neq 0$ at some $(t, x)$, write $p(\tau) = F(\tau, x)$ and $\tilde p(\tau) = \widetilde F(\tau,x)$, then
$$ D_t^k v = \bar\nabla_J^k \dot\gamma |_{s=0},$$
where $J$ is the variational vector field of the family of geodesics joining $p$ to $\tilde p$. From proposition \ref{proposition bounding nabla_J^k dot gamma}, we have
\begin{align*}
|D_t^k v| &\le C^{loc}_{k-1} d^2 P_{k-1}^2 (|H|, |D_t H|, \cdots , |D_t^k H|) + \\
&\quad + Q^2_{k-1} (|P\widetilde H - H|, |P D_t \widetilde H - D_t H|, \cdots, |PD_t^k \widetilde H - D_t^k H|).
\end{align*}
{Thus it suffices to show that
\begin{equation} \label{difference of D_t ^k tilde H and D_t^k H}
\lim_{i\to \infty} (P D_t^m \widetilde H - D_t^m H)(\tau_i, x_i) \to 0,
\end{equation}
for $m=0, 1, 2, \cdots$. By using (\ref{de of g-1}), $H = g^{ij} \nabla^2_{ij} F$ and inductively applying (\ref{de of nabla^i A}), we see that for any $m\in\mathbb{N}$, 
\begin{align}\label{1111}
D_t^m H = Q(g,h,F,m)
\end{align}
where $Q$ is a quantity involving only the spatial derivatives of $F$, $g$, $h$ and the derivatives of $\overline{Rm_h}$ evaluating at $(x,t)$ and $F(x,t)$. Since $F$ and $\tilde F$ are smooth up to $t=0$, for $m=0$, we have 
$$ \lim_{i\to \infty} D_t H (\tau_i, x_i) = \lim_{i\to \infty} D_t \widetilde H (\tau_i, x_i)=Q(g_0,h,F_0,0)$$
where the right hand side is the value of $Q$ evaluating at $x$ and $F_0(x)$. Since $P$ is continuous and $P = \operatorname{Id}$ at $t =0$, we obtain (\ref{difference of D_t ^k tilde H and D_t^k H}) for $m=0$. We can argue similarly for all $m\in\mathbb{N}$.}

\end{proof}


\section{Estimates for the parallel transport $P$}

Next we estimate the norm of the parallel transport $P$. We follow the same notations as in the previous section. First we prove the following lemma. 

\begin{lem}\label{metrics}
Let $F, \widetilde F$ be two families of immersions which satisfy (\ref{uniformly equivalence metric}). Then 
\begin{align}
\label{difference of g bounded by difference of DF} |g-\tilde g|&\leq C |P\tilde F_*-F_*|, \\
\label{difference of H bounded by difference of A} |P\widetilde H -H|& \le C(|P\tilde F_*-F_*| |\widetilde A|+ |P\widetilde A - A|).
\end{align}
\end{lem}

\begin{proof}
Since $g = h (F_*,F_*)$, together with Lemma \ref{lemma DF bounded},
\begin{align*}
|g-\tilde g| &= |h(F_*, F_*)-h (\widetilde F_* , \widetilde F_*)| \\
&=  |h( F_* ,F_*) -h(P\widetilde F_* , P\widetilde F_*)|\\
&= |h(  P\widetilde F_*-F_* , P\widetilde F_*) + h( F_* ,P\widetilde F_*-F_*)|\\
&\le C(|F_*|+|P\widetilde F_*|) |P\widetilde F_* - F_*|.
\end{align*}
Using the same argument and $H = g^{ij} A_{ij}$,
$$|P\widetilde H - H| \le C(|g^{-1}-\tilde g^{-1}| |\widetilde A| + |P\widetilde A -A|).$$
Thus (\ref{difference of H bounded by difference of A}) follows from (\ref{difference of g bounded by difference of DF}).
\end{proof}

\begin{prop} \label{proposition bound P}
Let $F, \widetilde F$ be two families of immersions so that (\ref{assumption on d}) holds. Then for any $p,q$, the parallel transport 
$$ P : \Gamma(\Sigma ,T^{p,q} \Sigma \otimes \widetilde N) \to \Gamma (\Sigma, T^{p,q} \Sigma \otimes N)$$
satisfy 
\begin{align}
\label{bound P} |P| =|P^{-1}| &=1,\\
\label{bound nabla P} |\nabla P|, |\nabla P^{-1}| & \le  C^{loc}_0 d
\end{align}
and
\begin{align}
\label{Delta P} |\Delta P| & \le C^{loc}_1\big(d\cdot \tilde P(|A|) + |P\widetilde F_* - F_ *| + | \Gamma - \widetilde \Gamma|+|P\widetilde A -A|\big)
\end{align}
for some polynomial $\tilde P$. {Here $\Delta P$ is the rough Laplacian of $P$ with respect to the connection $\nabla$: $\Delta P = g^{ij} \nabla_i \nabla _j P$.} 
\end{prop}

\begin{proof}
Since $P$ is given by $P(X\otimes Z) = X \otimes PZ$ for all $ X\in \Gamma (\Sigma, T^{p,q}\Sigma)$ and $Z\in \Gamma (\Sigma, \widetilde N)$, one can without loss of generality assume that $p=q=0$. 

Since parallel transport preserves length, $|PZ| = |Z|$ for all $Z \in T_{\tilde p}M$ and thus $|P| = |P^{-1}| =1$. To show (\ref{bound nabla P}), note again it suffices to assume $d\neq 0$. Let $x = (x^1, \cdots, x^n) \in \Sigma$, $i \in\{ 1, \cdots, n\}$ and $Z\in T_{\tilde p}M$ be fixed. Parallel transport $Z$ along the curve in $M$ with tangent vector $\widetilde F_i$. Thus $\nabla_{i} Z = 0$ and by the Leibniz rule,
\begin{equation*}
(\nabla_i P) (Z) = \nabla_i(PZ) - P(\nabla_i Z )= \nabla_i (PZ).
\end{equation*}
{Let $X$ be the parallel vector field along $\gamma$ with $X(1) = Z$. Then $X(0) = PZ$ by definition of $P$. By the definition of the connection on $N$, we have $\nabla_i (PZ) = \overline\nabla_{F_i} X(0)$. Lastly, since $F_i = J_i(0)$, we have $\nabla_i (PZ) = (\bar\nabla_{J_i} X)(0)$}. Differentiating the parallel transport equation gives
\begin{equation*}
\bar\nabla_{\dot\gamma} \bar\nabla_{J_i} X = - \bar R(\dot\gamma , J_i ) X, \ \ \ \nabla_{J_i} X (1) = \nabla_i Z  = 0.
\end{equation*}
By (\ref{bounding J}), we have
\begin{equation} \label{bound nabla_ dot gamma nabla _J_i X}
|\bar\nabla_{\dot\gamma} \bar\nabla_{J_i} X| \le C^{loc}_0 d (|F_i| + |P\widetilde F_i - F_i|)  |Z| .
\end{equation}
Together with lemma \ref{lemma DF bounded} this implies
\begin{equation*}
|(\nabla_iP) Z| = |\nabla_i (PZ)| \le C^{loc}_0 d|Z|.
\end{equation*}
Since this holds for all $Z$, we obtain (\ref{bound nabla P}). To show (\ref{Delta P}), we calculate under the normal coordinate at $x$ in $(\Sigma, g_t)$. Thus $\Delta P =\sum_i  \nabla_i \nabla_i P$. Now for each fixed $i$ and $Z\in T_{\tilde p}M$, we again parallel transport $Z$ along the curve in $M$ with tangent vector $\widetilde F_i$. Thus 
\begin{align*}
(\nabla_i \nabla_i P) (Z) &= \nabla_i [(\nabla_i P)(Z)] - \nabla P (\nabla_{e_i}e_i, Z) - (\nabla_i P)(\nabla_i Z) \\
&= \nabla_i (\nabla_i (PZ)) - \nabla_i (P (\nabla_i Z))- \nabla P (\nabla_{e_i}e_i, Z) - (\nabla_i P)(\nabla_i Z) \\
&= \nabla_i (\nabla_i (PZ))\\
&= \bar\nabla_{J_i} \bar\nabla_{J_i} X (0),
\end{align*}
where again $X$ is the parallel vector field along $\gamma$ with $X(1) = Z$. Since
\begin{equation*}
\bar\nabla_{\dot\gamma} \bar\nabla_{J_i} \bar\nabla_{J_i} X = -\bar R(\dot\gamma, J_i) \bar\nabla_{J_i}X - \bar\nabla _{J_i}  (\bar R(\dot\gamma, J_i ) X),
\end{equation*}
using (\ref{bound nabla_ dot gamma nabla _J_i X}), $\bar\nabla_{J_i} \bar\nabla_{J_i} X(1) = 0$, we obtain
\begin{equation*}
|2 \bar R (\dot\gamma, J_i) \bar\nabla_{J_i} X| \le C^{loc}_0 d^2 |Z|.
\end{equation*}
On the other hand, if we trace the induction proof of proposition \ref{proposition estimate D_v D_J^k J} and use (\ref{difference of H bounded by difference of A}), we have
\begin{align*}
|\bar\nabla_{J_i} J_i| &\le |\bar\nabla_{F_i} F_i| + |\bar\nabla_{\dot\gamma} \bar\nabla_{J_i} J_i|\\
&\le |A_{ii}|+ C^{loc}_0 d^2 (P(|F_i|, |A_{ii}|)+ C(|P\widetilde F_* - F_*| +|P \bar\nabla_{\widetilde F_i } \widetilde F_i - \bar\nabla_{F_i} F_i|)\\
&\le |A|+ C^{loc}_0 d^2 P(1, |A|) + C (|P\widetilde F_* - F_*| +|P\widetilde A_{ii} -A_{ii}| + | \widetilde\Gamma _{ii} ^k \widetilde F_k|) \\
&\le C^{loc}_0 \big(\tilde P(|A|) +|P\widetilde F_* - F_ *| + | \Gamma - \widetilde \Gamma|+|P\widetilde A -A|\big)
\end{align*}
 for some polynomial $\tilde P$. This implies
\begin{equation*}
|\bar R(\dot\gamma, \bar\nabla_{J_i} J_i)X| \le C^{loc}_0 d |Z|\big(\tilde P(|A|) +|P\widetilde F_* - F_ *| + | \Gamma - \widetilde \Gamma|+|P\widetilde A -A|\big).
\end{equation*}
Thus we have
\begin{equation*}
|(\nabla_i \nabla_i P) (Z)| \le C^{loc}_1 \big(d \tilde P(|A|) + |P\widetilde F_* - F_ *| + | \Gamma - \widetilde \Gamma|+|P\widetilde A -A|\big)|Z|
\end{equation*}
and this gives (\ref{Delta P}).
\end{proof}

The proof of the following proposition is similar to that of proposition \ref{proposition bound P} and is skipped.

\begin{prop} \label{proposition bound D_t P}
We have the estimates
\begin{equation} \label{derivative estimates of P}
|D_t P| \le  C^{loc}_0d ( |F_t|+ |P\widetilde F_t-F_t| ).
\end{equation}
\end{prop}

Next we derive an estimates for higher time covariant derivatives of $P$.

\begin{prop} \label{proposition higher time derivative estimates of P}
The $k$-th time derivatives of $P$ satisfies 
\begin{align} \label{higher time derivative estimates of P}
|D_t^k P| &\le C^{loc}_{k-1} d P^3_{k-1} (| F_t| , |D_t F_t| , \cdots , |D_{k-1} F_t|) +\\
&\quad + Q^3_{k-1} (|P\widetilde F_t - F_t|, \cdots, |P D_t^{k-1} \widetilde F_t -D_t^{k-1} F_t|), \nonumber
\end{align}
where $P_{k-1}^3, Q_{k-1}^3$ are universal polynomials of $k$-variables.
\end{prop}

In particular, we have
\begin{thm} \label{D_t^k P =0 at t=0}
Let $F, \widetilde F$ be smooth MCFs starting at $F_0$, then
\begin{equation}
D_t^k P \big|_{t=0} = 0
\end{equation}
for $k=1, 2, \cdots$.
\end{thm}

Now we prove proposition \ref{proposition higher time derivative estimates of P}.

\begin{proof}
As in the proof of proposition \ref{proposition bound P}, let $Z \in T_{\tilde p}M$ and extend it to a parallel vector fields along the integral curve of $\widetilde F_t$. Then
$$ (D^k_tP )(Z) = D_t^k (PZ).$$
Note that $D_t^k (PZ) = \bar\nabla_{J_t}^k X(0)$, where $X$ is the parallel transport of $Z$ along $\gamma$. We will prove by induction that
\begin{align} \label{Induction step for higher time derivative estimates of P}
|\bar\nabla_{J_t}^k X| &\le \bigg( (C^{loc}_{k-1} d P^3_{k-1} (| F_t| , |D_t F_t| , \cdots , |D^{k-1}_t F_t|) +\\
&\quad + Q^3_{k-1} (|P\widetilde F_t - F_t|, \cdots, |P D_t^{k-1} \widetilde F_t -D_t^{k-1} F_t|) \bigg) |Z|, \nonumber
\end{align}
where $P^3_{k-1}, Q^3_{k-1}$ are universal polynomials of $k$-variables.

When $k=1$, one can show as in the proof of proposition \ref{proposition bound P} the following estimates:
\begin{equation}
|\bar\nabla_{J_t} X | \le \big( C^{loc}_0 d |F_t| + C|P\widetilde F_t - F_t| \big) |Z|.
\end{equation}
Next we assume that (\ref{Induction step for higher time derivative estimates of P}) holds for all integers strictly smaller then $k$. by commuting covariant derivatives, we have
\begin{align*}
 \bar\nabla_{\dot\gamma} \bar\nabla_{J_t}^k X&= \sum_{i=0}^{k-1} \bar\nabla_{J_t}^i (\bar R(J_t, \dot\gamma) \bar\nabla_{J_t}^{k-1-i} X) \\
 &= \sum_{j+i+l+m = k-1}( \bar\nabla_{J_t}^j \bar R) * (\bar\nabla_{J_t}^iJ_t) *( \bar\nabla_{J_t}^l\dot\gamma )* (\bar\nabla_{J_t} ^{m} X).
\end{align*}
Thus the induction step is proved using $|\bar\nabla_{J_t}^k J_t| \le |D_t F_t| + \sup|\bar\nabla_{\dot\gamma} \nabla_{J_t} ^kJ_t|$, proposition \ref{proposition estimate D_v D_J^k J}, proposition \ref{proposition bounding nabla_J^k dot gamma}, the induction hypothesis and Cauchy Schwarz inequality. This finishes the proof of the proposition.
\end{proof}

Next, we prove the following lemma which estimates the difference of the restriction of ambient tensors to $F$ and $\widetilde F$. Let $S$ be a $(p,q)$-tensor on $M$. Then $S|_F$ is a section of the bundle $N^{\otimes p} \otimes (N^*)^{\otimes q}$ over $\Sigma$. Let $P^*(S |_{\widetilde F})$ be given by
\begin{equation*}
P^*S |_{\widetilde F} (a_1, \cdots, a_p, b^1 \cdots, b^q) = S|_{\widetilde F} ( P^* a_1 ,\cdots, P^* a_p, P^{-1} b^1, \cdots, P^{-1} b^q)
\end{equation*}
for all $a_i \in N^*, b^j \in N$.

\begin{lem} \label{lemma bound difference of ambient tensor}
With the above definition,
\begin{equation*}
| P^* (S|_{\widetilde F}) - S|_{F}| \le \sup |\bar \nabla S| \cdot d.
\end{equation*}
\end{lem}

\begin{proof}
Let $a_i, b^j$ be arbitrary and $a_i(s), b^j(s)$ be the respective parallel transport along $- \gamma$. Then by the fundamental theorem of calculus,
\begin{align*}
(P^*S|_{\widetilde F} - S|_{F}) (a_1, \cdots, a_p, b^1 , \cdots, b^q) &= \int_0^1 \partial_s \big( S(a_1(s), \cdots, a_p(s), b^1(s), \cdots, b^q(s)\big) ds \\
&=-\int_0^1 (\bar\nabla _{\dot\gamma} S)( a_1(s), \cdots, a_p(s), b^1(s), \cdots, b^q(s)) ds
\end{align*}
since $a_i$ and $b^j$'s are parallel along $-\gamma$. Thus
\begin{equation*}
|(P^*S|_{\widetilde F} - S|_{F}) (a_1, \cdots, a_p, b^1 , \cdots, b^q) | \le \sup |\bar\nabla S| |a_1| \cdots |a_p| |b^1 | \cdots |b^q| |\dot\gamma|.
\end{equation*}
Since $d = |\dot\gamma|$, the lemma is shown.
\end{proof}


\section{Main estimates}
In this section we provide the necessary estimates for the next two sections. In this section, we assume that $F, \widetilde F$ are both solutions to the MCF starting at $F_0$ which satisfies  (\ref{uniformly equivalence metric}) and (\ref{F C^0  close to F_0}). In particular, by choosing a small $T$, we assume that $d$ satisfies (\ref{assumption on d}).


First we estimate the time derivative of the quantities $P\widetilde F_* - F_*$ and $\Gamma - \widetilde \Gamma$.

\begin{lem} \label{lemma partial_t of difference of DF}
We have
\begin{equation}\label{partial_t of difference of DF}
|D_t (P\widetilde F_*  -F_*)| \le C_0 (|A| + |\widetilde A|) d + C|\widetilde \nabla\widetilde A| |P\widetilde F_* - F_*|+ C|P\widetilde\nabla\widetilde A - \nabla A|.
\end{equation}
\end{lem}
\begin{proof}Recall that
\begin{align*}
D_t (P\widetilde F_i-F_i)
&=(D_tP)\widetilde F_i +PD_t\widetilde F_i -D_t F_i\\
&=(D_tP)\widetilde F_i +P \widetilde \nabla_i \widetilde H-\nabla_i H\\
&=(D_t P)\widetilde F_i+(\tilde g^{kl}-g^{kl})(P\widetilde\nabla_i \widetilde A_{kl})+g^{kl}(P\widetilde\nabla_i \widetilde A_{kl}-\nabla_i A_{kl}).
\end{align*}
Using Proposition \ref{proposition bound D_t P}, Lemma \ref{metrics} and (\ref{uniformly equivalence metric}), the result follows.
\end{proof}

To estimate the time derivative of $\Gamma- \widetilde \Gamma$. From (\ref{de of Gamma ijk}) we have
\begin{equation*}
\partial_t \Gamma = g^{-2}  * h(A , \nabla A),
\end{equation*}
Thus $|\partial_t (\Gamma- \widetilde \Gamma)|$ can be estimated as in the proof of Lemma \ref{lemma partial_t of difference of DF} and (\ref{difference of g bounded by difference of DF}). We skip the proof of the following lemma:

\begin{lem}
\begin{equation} \label{partial_t of difference of Gamma}
|\partial_t (\Gamma-\widetilde \Gamma)| \leq C|A| |\nabla A| |P\widetilde F_* - F_*|+ C |\nabla A| |P\widetilde A - A| + C |\widetilde A| |P\widetilde \nabla\widetilde A - \nabla A|.
\end{equation}
\end{lem}

Next let us consider the second order quantity
\begin{equation} \label{norm of P tilde A - A}
|P\widetilde A - A|^2= g^{ik} g^{jl} h( P\widetilde A_{ij} -A_{ij}, P\widetilde A_{kl} - A_{kl}).
\end{equation}

\begin{prop} \label{prop partial_t of difference of A}
We have the estimate
\begin{equation} \label{partial_t of difference of A}
\begin{split}
&\quad  \partial_t |P\widetilde A-A|^2-2h(P\widetilde\Delta \widetilde A-\Delta A,P\widetilde A-A)\\
&\leq \frac{1}{6}|P\widetilde \nabla\widetilde A-\nabla A|^2+ C ( (B_2^{loc})^2 + C_1^{loc} + |A|^2 + |\widetilde A|^2) d^2 \\
&\quad + C(|\widetilde A|^4 + |\widetilde \nabla\widetilde A|^2 + C_1^{loc} )|P\widetilde F_*  - F_*|^2 +C (|A|^2 + |\widetilde A|^2 +C_0^{loc} )|P\widetilde A - A|^2.
\end{split}
\end{equation}
\end{prop}

\begin{proof}
From (\ref{norm of P tilde A - A}) and (\ref{de of g-1}) and proposition \ref{proposition bound D_t P},
\begin{align*}
\partial_t |P\widetilde A - A|^2 &= \partial_t\big( g^{ik} g^{jl} h( P\widetilde A_{ij} -A_{ij}, P\widetilde A_{kl} - A_{kl})\big) \\
&= 2 (\partial_t g^{ik}) g^{jl} h( P\widetilde A_{ij} -A_{ij}, P\widetilde A_{kl} - A_{kl})\big) \\
&\quad+ 2 g^{ik} g^{jl}h( D_t (P\widetilde A_{ij} -A_{ij}), P\widetilde A_{kl} - A_{kl})\\
&\le C|A|^2 |P\widetilde A - A|^2 + C^{loc}_0 d(|A|^2 + |\widetilde A|^2) |P\widetilde A - A| \\
&\quad + 2 g^{ik} g^{jl}h(  PD_t \widetilde A_{ij} -D_tA_{ij}, P\widetilde A_{kl} - A_{kl}).
\end{align*}
Now use (\ref{de of nabla^i A}) with $k=2$ to write
\begin{align*}
PD_t \widetilde A - D_t A = P\widetilde\Delta \widetilde A - \Delta A+(I)+(II)+(III)+(IV)
\end{align*}
where
\begin{align*}
(I)&= P(\tilde g^{-2}*\tilde h(\widetilde A,\widetilde A)*\widetilde A)- g^{-2}*h(A,A)*A;\\
(II)&=P(\tilde g^{-2}*\tilde h(\widetilde \nabla\widetilde A,\widetilde A)*\widetilde F_*)- g^{-2}*h(\nabla A,A)*F_*;\\
(III)&= P (\bar\nabla \bar R|_{\widetilde F} * (\widetilde F_*)^{a+2} * \tilde g^{-b}) - \bar\nabla \bar R|_{ F} * ( F_*)^{a+2} * g^{-b};\\
(IV)&=P (\bar R|_{\widetilde F} * (\widetilde F_*)^{a} *\widetilde A * \tilde g^{-b}) - \bar R|_{ F} * ( F_*)^{a} * A*g^{-b}.
\end{align*}
Using $\tilde h (\widetilde A, \widetilde A) = h(P\widetilde A, P\widetilde A)$, one has
\begin{align*}
(I)&= (g^{-1}-\tilde g^{-1})*\tilde g^{-1} * h(P\widetilde A,P\widetilde A)*P\widetilde A \\
&\quad + g^{-1}*(g^{-1}-\tilde g^{-1}) * h(P\widetilde A,P\widetilde A)*P\widetilde A  \\
&\quad + g^{-2}* \big( h(P\widetilde A-A,P\widetilde A)+ h(A, P\widetilde A - A)\big) *P\widetilde A \\
&\quad + g^{-2} * h(A, A) * (P\widetilde A - A)\\
\Rightarrow |(I)| &\le C |\widetilde A|^3 |P\widetilde F_* - F_*| + C (|A|^2 + |\widetilde A|^2) |P\widetilde A - A|.
\end{align*}
Similarly we have
\begin{align*}
|(II)| \le C |\widetilde A||\widetilde \nabla\widetilde A| |P\widetilde F_* - F_*| +C |\widetilde A| |P\widetilde A - A| + C|A||P\widetilde \nabla \widetilde A - \nabla A|.
\end{align*}
For $(III)$, note
$$ P (\bar\nabla \bar R|_{\widetilde F} * (\widetilde F_*)^{a+2} * \tilde g^{-b}) = (P^*\bar\nabla \bar R|_{\widetilde F} * (P\widetilde F_*)^{a+2} * \tilde g^{-b}).$$
Thus a similar calculation and lemma \ref{lemma bound difference of ambient tensor} give
\begin{align*}
|(III)| \le C B^{loc}_2 d + C^{loc}_1 |P\widetilde F_* - F_*|.
\end{align*}
Similar for $(IV)$ we have
\begin{align*}
|(IV)| \le C^{loc}_1 |\widetilde A| d + C^{loc}_0 |\widetilde A| |P\widetilde F_* - F_*| + C^{loc}_0 |P\widetilde A - A|.
\end{align*}
Therefore
\begin{align*}
\partial_t |P\widetilde A - A|^2 &\le 2h( (P\widetilde \Delta \widetilde A -\Delta A), P\widetilde A - A) + C_0^{loc} (|A|^2 + |\widetilde A|^2) d^2 \\
&\quad+ C(|A|^2 + |\widetilde A|^2) |P\widetilde A - A|^2 \\
&\quad + C \bigg[ (|\widetilde A||\widetilde\nabla\widetilde A| + |\widetilde A|^3 + C_1^{loc} + C_0^{loc} |\widetilde A|) |P\widetilde F_* - F_*|\\
&\quad + (|A|^2 + |\widetilde A|^2+ C_0^{loc} ) |P\widetilde A- A| + |A| |P\widetilde\nabla\widetilde A - \nabla A| \\
&\quad+ (C B_2^{loc} + C_1^{loc} |\widetilde A|)d\bigg] |P\widetilde A- A|.
\end{align*}
Now (\ref{partial_t of difference of A}) is obtained using Cauchy Schwarz inequalities.
\end{proof}

\section{Proof of theorem \ref{Uniqueness theorem}, theorem \ref{Uniqueness theorem for t^alpha bound} and theorem \ref{RicciFlow}}
In this section, we use the energy argument to prove the theorem \ref{Uniqueness theorem}, theorem \ref{Uniqueness theorem for t^alpha bound} and theorem \ref{RicciFlow}. 

To prove theorem \ref{Uniqueness theorem} using the energy method, we introduce the following energy quantity. By \cite{GreeneWu}, we can find $\rho\in C^\infty(M)$ such that $|\bar\nabla \rho|\leq 2$ and
$$d_M(\cdot,y_0)\leq \rho(\cdot)\leq d_M(\cdot,y_0)+1.$$
Now define
$$\mathcal{Q}=d^2+|\Gamma-\widetilde \Gamma|^2+\frac{\rho^{2-\epsilon}(F_0)}{t}|P\widetilde F_*-F_*|^2 + |P\widetilde A-A|^2, \ \ \ t>0.$$

\begin{lem} \label{partial_t of Q}
Under the assumption of the theorem \ref{Uniqueness theorem}, there exists $C_1$ such that on $F_0^{-1}\left(B_M(p,r) \right)$, where  $r>>1$,
$$\partial_t \mathcal{Q}\leq \frac{C_1r^{2-\epsilon}}{t}\mathcal{Q}+2 h(P\widetilde\Delta \widetilde A - \Delta A, P\widetilde A- A)+\frac{1}{2}|P\tilde \nabla\tilde A-\nabla A|^2.$$
\end{lem}

\begin{proof}
Recall that from lemma \ref{lemma derivative of d}, (\ref{difference of H bounded by difference of A}) and (\ref{assumption on A}), we have
\begin{align*}
\partial_t d^2 \leq Cd(|\widetilde A| |P\widetilde F_*  - F_*| + |P\widetilde A - A|) \le \frac{Cr^{2-\epsilon}}{t} \mathcal{Q},
\end{align*}
On the other hand, using (\ref{partial_t of difference of DF}), (\ref{partial_t of difference of Gamma}) and (\ref{assumption on A}), (\ref{first order estimates}) give
\begin{align*}
\partial_t |\widetilde \Gamma-\Gamma|^2
&\leq C|\widetilde \Gamma-\Gamma|\left[|A||\nabla A||P\widetilde F_*-F_*|+|\nabla A||P\widetilde A-A|+|\widetilde A||P\widetilde\nabla\widetilde A-\nabla A|\right]\\
&\quad +C|A|^2|\widetilde \Gamma-\Gamma|^2\\
&\leq \frac{C_1r^{2-\epsilon}}{t}\mathcal{Q}+\frac{1}{6}|P\widetilde \nabla\widetilde A-\nabla A|^2
\end{align*}
and
\begin{align*}
\frac{\partial}{\partial t} \left(\frac{\rho^{2-\epsilon}}{t}|P\widetilde F_*-F_*|^2\right)
&\leq \frac{C_1r^{2-\epsilon}}{t}\mathcal{Q}+\frac{1}{6}|P\widetilde \nabla\widetilde A-\nabla A|^2.
\end{align*}
Lastly, using $(B_2^{loc})^2 \le C (1+ r^{2-\epsilon})$ for $r>>1$ and (\ref{partial_t of difference of A}), we have
$$ \partial_t |P\widetilde A-  A|^2 \le 2 h(P\widetilde\Delta \widetilde A - \Delta A, P\widetilde A- A) + C_1 \frac{r^{2-\epsilon}}{t} \mathcal{Q} + \frac{1}{6} |P\widetilde \nabla\widetilde A- \nabla A|^2$$
and the lemma is proved.
\end{proof}

\begin{proof}[Proof of theorem \ref{Uniqueness theorem}]
For each $r>>1$, let $\phi(x)=\varphi^p(\rho(F_0)/r)$ where $\varphi$ is smooth, equals $1$ on $[0,1/2]$, vanishes outside $[0,1]$ and satisfies $0\leq -\varphi'\leq 10$. Here $p$ possibly depends on $r$. Let $\eta(x,t)= \frac{[\rho (F_0)]^2}{a-bt}$. For $t\in (0,a/(2b)]$, $\eta(x,t)\geq \rho^2/a$. Moreover,
\begin{align*}
\partial_t \eta= \frac{b}{(a-bt)^2} \rho^2 \geq \frac{b}{8n\lambda} |\nabla \eta|^2.
\end{align*}
Here $a$ and $b$ are some constants to be fixed later and $\lambda$ is a constant such that $g(t)\geq \lambda^{-1} g_0$. Now for $t\in (0,a/2b]$, define the energy $E_r(t)$ as
\begin{align*}
E_r(t)=\int_\Sigma \mathcal{Q}\xc\,d\mu,\;\;\text{for}\;t>0.
\end{align*}
Note that the above is well defined, since $\phi$ is of compact support, while (\ref{F C^0  close to F_0}) and the properness of $F_0$ together imply that $F(t,\cdot)$ is also proper. From lemma \ref{partial_t of Q} and $\partial_t d\mu \le 0$ by (\ref{de of dmu}),
\begin{align*}
 \frac{d}{dt}E_r(t)
&\leq -\int_\Sigma \mathcal{Q} \xc \partial_t\eta\,d\mu+\frac{C_1r^{2-\epsilon}}{t}E_r(t)+ \int_\Sigma |P\widetilde \nabla \widetilde A  -\nabla A|^2 \xc d\mu\\
&\quad+2\int_\Sigma  h(P\widetilde \Delta \widetilde A - \Delta A, P\widetilde A - A)\xc \; d \mu.
\end{align*}
We focus on the term containing the Laplacians. Using (\ref{bound nabla P}), lemma \ref{lemma first order estimates} and Cauchy Schwarz inequality,
\begin{align*}
&\quad 2\int_\Sigma  h(P\widetilde \Delta \widetilde A - \Delta A, P\widetilde A - A)\xc \; d\mu \\
&=2\int_\Sigma h\left( P \tilde g^{ij} \widetilde\nabla_i\widetilde \nabla_j \widetilde A - g^{ij} \nabla_i \nabla_j A, P\widetilde A - A\right)\xc \; d\mu\\
&=2 \int_\Sigma h\left(  -\tilde g^{ij} (\nabla_i P) \widetilde \nabla_j \widetilde A +\widetilde\nabla_i (P\tilde g^{ij} \widetilde \nabla_j \widetilde A)- g^{ij} \nabla_i \nabla_j A, P\widetilde A - A\right) \xc \; d\mu \\
&\le C_0 \int_\Sigma d |\widetilde \nabla \widetilde A| |P\widetilde A - A| \xc \; d\mu +  2\int_\Sigma h\left( \nabla_i \big(P\tilde g^{ij}\widetilde \nabla_j \widetilde A- g^{ij}  \nabla_j A\big), P\widetilde A - A\right) \xc \; d\mu\\
&\quad +C \int_\Sigma |\Gamma-\widetilde \Gamma| |\widetilde \nabla \widetilde A| |P\widetilde A - A| \xc \; d \mu \\
&\le \frac{C_1r^{2-\epsilon}}{t} E_r(t) +   2\int_\Sigma h\left( \nabla_i \big(P\tilde g^{ij}\widetilde \nabla_j \widetilde A- g^{ij} \nabla_j A\big), P\widetilde A - A\right) \xc \; d\mu.
\end{align*}
Now we use integration by part to the second term on the right hand side to obtain
\begin{align*}
&\quad  2\int_\Sigma h\left( \nabla_i \big(P\tilde g^{ij}\widetilde \nabla_j \widetilde A- g^{ij} \nabla_j A\big), P\widetilde A - A\right) \xc \; d\mu \\
&= -2  \int_\Sigma h\left( P\tilde g^{ij}\widetilde \nabla_j \widetilde A- g^{ij} \nabla_j A, (\nabla_i P)\widetilde A\right) \xc \; d\mu\\
&\quad -2 \int_\Sigma h\left( P\tilde g^{ij}\widetilde \nabla_j \widetilde A- g^{ij} \nabla_j A, P(\nabla_i \widetilde A) -\nabla_i A\right) \xc \; d\mu\\
&\quad -2  \int_\Sigma h\left( P\tilde g^{ij}\widetilde \nabla_j \widetilde A- g^{ij} \nabla_j A,  P\widetilde A - A\right) \nabla_i(\xc) \; d\mu \\
&= (A) +(B) +(C).
\end{align*}
For the first two terms, we use again
\[ \tilde g^{-1} = \tilde g^{-1} - g^{-1} + g^{-1} , \ \ \ \nabla_i  = \nabla_i - \widetilde \nabla_i + \widetilde \nabla_i,\]
proposition \ref{proposition bound P} and Cauchy Schwarz inequality to get
\begin{equation*}
(A) +(B) \le \frac{C_1r^{2-\epsilon}}{t} E_r(t) - \int_\Sigma |P\widetilde \nabla\widetilde A - \nabla A|^2 \xc \; d\mu.
\end{equation*}
For (C) we have similarly
\begin{align*}
(C)&\leq \frac{1}{2}\int_\Sigma |P\widetilde\nabla \widetilde A-\nabla A|^2e^{-\eta}\phi d\mu+\frac{C_1r^{2-\epsilon}}{t}E_r(t)\\
&\quad +C\int_\Sigma |P\widetilde A-A|^2e^{-\eta} \left[\frac{|\nabla\phi|^2}{\phi}+\phi |\nabla\eta|^2\right]d\mu.
\end{align*}
Combine all these,
\begin{equation} \label{final estimates of partial_t E_R}
\begin{split}
\frac{d}{dt} E_r(t)&\leq \frac{C_1r^{2-\epsilon}}{t}E_r(t)+\int_\Sigma \mathcal{Q}\xc \left[-\frac{\partial}{\partial t}\eta +\tilde C|\nabla\eta|^2\right]\,d\mu \\
&\quad+C\int_{F_0^{-1}(A_{y_0}(r/2,r))} |P\widetilde A-A|^2e^{-\eta}\frac{|\nabla\phi|^2}{\phi}d\mu,
\end{split}
\end{equation}
where $A_{y_0} (r/2,r) \subset M$ is the annulus centred at $y_0$ and $\tilde C$ is a fixed constants depending only on the dimensions of $\Sigma, M, \lambda$ and $L$. To estimate the last term on the right hand side of (\ref{final estimates of partial_t E_R}), note
\begin{align*}
\frac{|\nabla \phi|^2}{\phi} \leq \frac{Cp^2 }{r^2}\phi ^{1-2/p}.
\end{align*}
By Young's inequality, (\ref{assumption on A}), $\partial_t d\mu\leq 0$ and (\ref{assumption on volume growth of F_0}), we have
\begin{align*}
&\quad \frac{Cp^2}{r^2}\int_{F_0^{-1}(A_{y_0}(r/2,r))} |P\widetilde A-A|^2e^{-\eta}\phi^{1-2/p}d\mu\\
&\leq \frac{Cp^2}{r^2}E_r^{1-2/p}\left(\int_{F_0^{-1}(A_{y_0}(r/2,r))}|P\widetilde A-A|^2e^{-\eta}d\mu\right)^{2/p}\\
&\leq \frac{r^{2-\epsilon}}{t}E_r+\frac{C^{p/2}p^{p}t^{p/2-1}}{r^{p+(2-\epsilon)(p/2-1)}}\cdot \left(\int_{F_0^{-1}(A_{y_0}(r/2,r))}(|A|^2+|\widetilde A|^2)e^{-\eta}d\mu\right)\\
&\leq  \frac{r^{2-\epsilon}}{t}E_r+ \frac{C^{p/2}p^{p}t^{p/2-2}}{r^{p+(2-\epsilon)(p/2-2)}} e^{-r^2/a} \cdot \int_{F_0^{-1}(A_{y_0}(r/2,r))} d\mu\\
&\leq  \frac{r^{2-\epsilon}}{t}E_r+ \frac{C^{p/2}p^{p}t^{p/2-2}}{r^{p+(2-\epsilon)(p/2-2)}} e^{-r^2/a} \cdot \int_{F_0^{-1}(A_{y_0}(r/2,r))} d\mu_0\\
&\leq \frac{r^{2-\epsilon}}{t}E_r+\frac{DC^{p/2}p^{p}t^{p/2-2}}{r^{p+(2-\epsilon)(p/2-2)}}e^{-r^2/a+Dr^2}.
\end{align*}

Now we require that $a,b$ satisfy $a^{-1}\geq 2D$ and $b/(8n\lambda)\geq \tilde C$. Therefore, the differential inequality for the energy quantity reduces to
\begin{align*}
\frac{d}{dt} E_r(t)\leq \frac{C_1r^{2-\epsilon}}{t}E_r(t)+D(\sqrt C p)^p t^{p/2-2} e^{-r^2/(2a)}.
\end{align*}
For each $r>>1$, from now on we consider $C_1, C$ as fixed constants and write $\alpha=C_1r^{2-\epsilon}>0$. Solve the above ode on $0<s<t<a/(2b)$:
\begin{align*}
\frac{E_r(t)}{t^\alpha}\leq \frac{E_r(s)}{s^\alpha}+D(\sqrt Cp)^p e^{-r^2/(2a)}\int^t_s x^{p/2-2-\alpha}\,dx
\end{align*}

At each $r>>1$, choose $p=2(2+\alpha) = 2(2+ C_1r^{2-\epsilon})$. For $r$ large (depending only on $C, D, \epsilon, a$) we have $e^{r^2/(4a)}>D(\sqrt Cp)^p$. Hence,
\begin{equation} \label{integrating the differential ineq}
\frac{E_r(t)}{t^\alpha}\leq \frac{E_r(s)}{s^\alpha}+\frac{a}{2b} e^{-r^2/(4a)}
\end{equation}
for large enough $r$. By theorem \ref{D_t^k P =0 at t=0} and the MCF equation, since the convergence $F(t,\cdot), \widetilde F(t,\cdot) \to F_0(\cdot)$ are smooth, $\mathcal{Q}^{(m)}(0)=0$ for any $m\in\mathbb{N}$. Since $F_0$ is proper and $M$ is complete, $F_0^{-1}(B_M(y_0,2r))$ is a compact set and we may apply the dominated convergent theorem to conclude that
$$\lim_{s\rightarrow 0}\frac{E_r(s)}{s^\alpha}=0.$$
Followed by letting $r\rightarrow \infty$ in (\ref{integrating the differential ineq}), we have $\mathcal Q \equiv 0$ for all $t\in [0,a/(2b)]$, in particular $F=\widetilde F$ in $[0,a/2b]$. Extension to the whole interval $[0,T]$ follows from an open-closed argument and this finishes the proof of theorem \ref{Uniqueness theorem}. 
\end{proof}

Next we prove theorem \ref{Uniqueness theorem for t^alpha bound}.

\begin{proof}[Proof of theorem \ref{Uniqueness theorem for t^alpha bound}]
In this situation, we observe that conditions (1) and (2) in theorem \ref{Uniqueness theorem} hold since $t^{-2\alpha}$ is integrable in $[0,T]$. Thus all the calculations in the sections 3, 4, and 5 can be applied. Let
$$ \mathcal Q^\alpha = d^2 + |\Gamma - \widetilde \Gamma|^2 + t^{-2\alpha} |P\widetilde F_* - F_*|^2 + |P\widetilde A - A|^2. $$
In this situation, one uses an intrinsic cutoff function: Let $\rho = \rho_T \in C^\infty (\Sigma)$ be an exhaustion of $(\Sigma, g_T)$ so that for some $x_0\in \Sigma$,
$$  d_{g_T} (x,x_0) \le \rho(x) \le  d_{g_T} (x,x_0)+1, \ \ \ |\nabla \rho |\le 2.$$
Let $\phi$, $\eta$ be defined as in the proof of theorem \ref{Uniqueness theorem} with this new $\rho$ and let
$$ E^\alpha_r(t) = \int_\Sigma \mathcal Q^\alpha \xc d\mu.$$
The assumption $|A|  + |\widetilde A| \le C/t^\alpha$ implies the estimates $|\nabla A| + |\widetilde \nabla\widetilde A| \le C_1/t^{2\alpha}$. Arguing as in the proof of theorem \ref{Uniqueness theorem}, we have for $r>>1$ and $b/(8n\lambda) \ge \tilde C$,
\begin{align*}
\frac{d}{dt} E^\alpha_r(t) \le \frac{C_2}{t^{2\alpha}} E^\alpha_r(t) + C \int_{F_0^{-1} A_{y_0} (r/2,r)} |P\widetilde A - A|^2 e^{-\eta} \frac{|\nabla \phi|^2}{\phi} d\mu.
\end{align*}
Using the assumption on $|A|, |\widetilde A|$, (\ref{uniformly equivalence metric}) and pick $p=2$, we have
\begin{align*}
\frac{d}{dt} E_r^{\alpha} (t) \le \frac{C_2}{t^{2\alpha}} \bigg( E_r^\alpha(t) + e^{-ar^2} \text{Vol}_{g_T} (B_r(x_0))\bigg).
\end{align*}
From the Gauss equation and the assumptions on $A$, $(\Sigma, g_T)$ has bounded curvature, thus the volume comparison theorem gives
$$ \text{Vol}_{g_T} (B_r(x_0)) \le D e^{Dr}$$
for some $D = D(n,m,T, B_0)$. Choosing $a^{-1} \ge 2D$,
\begin{align*}
\frac{d}{dt} E^\alpha_r(t)\le \frac{C_2}{t^{2\alpha}}\left( E^\alpha_r(t) + e^{-r^2/2a}\right).
\end{align*}
Since the convergence $F(t, \cdot), \widetilde F(t, \cdot) \to F_0(\cdot)$ is $C^3$, $E_r(t)$ is continuous at $t=0$ and $E_r(0) = 0$. Integrating the above inequality (note $t^{-2\alpha}$ is integrable) gives
$$E_r(t) \le \left( e^{\frac{C_2}{1-2\alpha} t^{1-2\alpha}} -1\right) e^{-r^2 /2a}.$$
Let $r\to \infty$ gives $\mathcal Q^\alpha = 0$ for all $t\in [0,a/2b]$. Thus $F = \widetilde F$ in $[0,a/2b]$ and the theorem follows from iterating the argument.
\end{proof}

Using the above cutoff technique and the argument in the proof of theorem \ref{Uniqueness theorem}, we sketch how one can prove theorem \ref{RicciFlow}.

\begin{proof}[Sketch of proof of theorem \ref{RicciFlow}]

We argue using similar argument in \cite{Kotschwar}. Define the energy to be
$$E_R(t)=\int_M \left(t^{-2}|g-\tilde g|^2+t^{-1}|\Gamma-\tilde \Gamma|^2+|Rm-\widetilde{Rm}|^2\right)\phi e^{-\eta}\,d\mu_{g(t)}.$$
Here we choose the cutoff function and exhaustion function as in the proof of theorem \ref{Uniqueness theorem}: $\phi(x)=\phi(\rho(x)/R)$ and $\eta(x,t)=\frac{\rho(x)^2}{a-bt}$ where $\rho$ is a smooth function on $M$ such that 
$$d_0(x,x_0)\leq \rho(x)\leq d_0(x,x_0)+1\quad\text{and}\quad |\nabla^{g_0} \rho|\leq 2$$
for some $x_0\in M$. By volume comparison and equivalence of metrics, we know that
$$V_t(B_0(p,R))\leq V_T(B_T(p,CR))\leq C'e^{C'R}.$$
Using integration by part, we obtain a evolution inequality of $E_R$ which is in the same form as before.
$$E_R'(t)\leq \frac{L}{t}E_R(t)+C_n\int_M \frac{|\nabla\phi|^2}{\phi}|Rm-\widetilde{Rm}|^2 e^{-\eta}\,d\mu$$
for some $L=L(n,\lambda)$. We can now employ the same trick in the proof of theorem \ref{Uniqueness theorem} to conclude that $g(t)=\tilde g(t)$ for all $t\in [0,T]$.
\end{proof}


\section{Backward Uniqueness}
In this section, we modify a general backward uniqueness result in \cite{Kotschwar4} to prove theorem \ref{Backward uniqueness theorem}. When the ambient space is Euclidean, similar results were obtained in \cite{Huang} in co-dimension one case and \cite{Zhang} in arbitrary co-dimension. However, the issue of parallel transport is not addressed in \cite{Huang}, \cite{Zhang} when the ambient space is not Euclidean.

To start the proof, let $F, \widetilde F : [0,T]\times \Sigma \to M$ be two MCFs with uniformly bounded second fundamental forms $|A|+ |\widetilde A|\le C$ and $F = \widetilde F$ at time $T$. To show backward uniqueness, it suffices to show that $F = \widetilde F$ on $[1/l, T]$ for all $l\in \mathbb N$. Now consider $l$ as fixed number. By theorem 3.2 in \cite{ChenYin}, we have
\begin{equation*}
|\nabla^k A|+|\widetilde\nabla^k\widetilde A| \le C_{k+1}, \ \ \ k=0,1,2,\cdots  \text{ and } t\in [1/l,T].
\end{equation*}

Consider two (time-dependent) vector bundles over $\Sigma$:
\begin{equation*}
\mathcal X = (T^{0,2}\Sigma \otimes N) \oplus (T^{0,3}\Sigma \otimes N) , \ \ \ \mathcal Y =N \oplus (T^{0,1} \Sigma \otimes N)\oplus T^{1,2}\Sigma \oplus T^{1,3}\Sigma.
\end{equation*}
We use the metric induced from $g$ and $h$ and the direct sums are orthogonal. Define the following time covariant derivatives on $\mathcal X$ and $\mathcal Y$ respectively: 
$$D_t ^{\mathcal X} = D_t \oplus D_t , \ \ \ D_t^{\mathcal Y} = D_t \oplus D_t \oplus \partial_t \oplus \partial_t.$$
Consider the following two sections $X, Y$ on $\mathcal X,\mathcal Y$ respectively:
\begin{equation*}
X =  (P\widetilde A - A)\oplus (P\widetilde \nabla \widetilde A- \nabla A), \ \ \ Y= v \oplus (P\widetilde F_* - F_*)\oplus (\Gamma - \widetilde \Gamma) \oplus \nabla (\Gamma - \widetilde \Gamma),
\end{equation*}
where $v$ is defined in section 3. Theorem \ref{Backward uniqueness theorem} follows from the following

\begin{thm} \label{theorem PDE-ODE ineq}
There are constants $C_4$ so that
\begin{align}
\label{PDE-ODE ineq 1} |(D_t^{\mathcal X} - \Delta)X| &\le C_4(|X| + |\nabla X|+ |Y|), \\
\label{PDE-ODE ineq 2} |D_t^{\mathcal Y} Y| &\le C_4(|X| + |\nabla X| + |Y|).
\end{align}
\end{thm}

\begin{proof}
First we estimate $\partial_t (\nabla (\Gamma - \widetilde \Gamma))$. We remark that for any $(p,q)$ tensors $S$ on $\Sigma$, we have
$$ (\partial_t \nabla - \nabla \partial_t )S = (\partial_t \Gamma) *S.$$
Then we have the estimates
\begin{align*}
|\partial_t \nabla (\Gamma - \widetilde \Gamma)| &= |\nabla (\partial_t \Gamma - \partial_t \widetilde \Gamma)| + C_2|\Gamma - \widetilde \Gamma| \\
&= \left|\nabla\big[ g^{-2} * h(A, \nabla A) - \tilde g^{-2} * \tilde h (\widetilde A, \widetilde \nabla \widetilde A)\big]\right| +C_2|\Gamma - \widetilde \Gamma| \\
&\le |g^{-2} * (h (\nabla A, \nabla A) + h(A, \nabla^2 A) )- \tilde g^{-2}  *( h (P\widetilde \nabla \widetilde A, P\widetilde \nabla \widetilde A)| \\
&\quad + h (P\widetilde A, P\widetilde \nabla ^2\widetilde A))+ C_2 |\Gamma - \widetilde \Gamma| \\
&\le C_2(|\Gamma - \widetilde \Gamma| +|P\widetilde F_* - F_*|+ |P\widetilde A -A| + |P\widetilde \nabla\widetilde A - \nabla A| + |P\widetilde \nabla^2 \widetilde A - \nabla^2 A|)\\
&\le  C_2(|\Gamma - \widetilde \Gamma| +|P\widetilde F_* - F_*|+ |P\widetilde A -A| + |P\widetilde \nabla\widetilde A - \nabla A| \\
&\quad + |\nabla(P\widetilde \nabla \widetilde A - \nabla A)|+ |v|) \\
&\le C_2 (|X| + |\nabla X| + |Y|).
\end{align*}
The above inequality together with (\ref{bound D_t v}), (\ref{partial_t of difference of DF}) and (\ref{partial_t of difference of Gamma}) give us (\ref{PDE-ODE ineq 2}). To derive (\ref{PDE-ODE ineq 1}), note that for any $k$,
\begin{align*}
(D_t - \Delta) (P\widetilde \nabla^k \widetilde A - \nabla^k A) &= ((D_t-\Delta )P) \widetilde\nabla^k \widetilde A - 2g^{ij} (\nabla_i P)(\nabla_j \widetilde\nabla^k \widetilde A) \\
&\quad + P( (D_t - \Delta)\widetilde \nabla^k \widetilde A) - (D_t - \Delta)\nabla^k A.
\end{align*}
The first two terms on the right hand side is estimated using (\ref{derivative estimates of P}), (\ref{Delta P}) and  (\ref{bound nabla P}):
\begin{align*}
| ((D_t - \Delta )P ) \widetilde \nabla^k \widetilde A|&\le C_1|\widetilde\nabla^k \widetilde A|[ (|A|+|\widetilde A|) |v| + \tilde P(|A|) |v| + \\
&\quad |P\widetilde F_* -F_*| + |\Gamma - \widetilde \Gamma| + |P\widetilde A-A|] \\
&\le C_{k+1} (|X| + |Y|), \\
|2g^{ij} (\nabla _i P) (\nabla_j \widetilde\nabla^k \widetilde A)| &\le C_{k+2}|Y|.
\end{align*}
To estimate the third term we use
$$\Delta -\widetilde \Delta = g^{-1} * \nabla(\Gamma - \widetilde \Gamma)  + g^{-1}* (\Gamma- \widetilde \Gamma)* \widetilde \nabla + (g^{-1} - \tilde g^{-1})* \widetilde\nabla^2 $$
and get
\begin{align*}
|(D_t - \Delta) (P\widetilde \nabla^k \widetilde A - \nabla^k A)| &\le |P(D_t -\widetilde \Delta)\widetilde\nabla^k \widetilde A - (D_t - \Delta) \nabla^k A|\\
&\quad + C\big[ (|\widetilde \nabla^k\widetilde A| + |\widetilde\nabla^{k+1} \widetilde A|) |v|+ |\widetilde\nabla^k \widetilde A| |\nabla(\Gamma - \widetilde \Gamma)| \\
&\quad + (|\widetilde\nabla^k \widetilde A|+|\widetilde\nabla^{k+1} \widetilde A| )|\Gamma - \widetilde \Gamma| + |\widetilde\nabla^{k+2} \widetilde A| |P\widetilde F_* - F_*| \big] \\
&\le |P(D_t -\widetilde \Delta)\widetilde\nabla^k \widetilde A - (D_t - \Delta) \nabla^k A| + C_{k+3} (|X| + |Y|).
\end{align*}
From (\ref{de of nabla^i A}), one can check that
\begin{align*}
|P(D_t -\widetilde \Delta)&\widetilde\nabla^k \widetilde A - (D_t - \Delta) \nabla^k A|\\
& \le C_{k+2} \left(|P\widetilde A -A|+\sum_{i=0}^{k} |\nabla (P\widetilde\nabla^i \widetilde A - \nabla^i A)| + |\Gamma - \widetilde \Gamma|+ |P\widetilde F_* - F_*| + |v|\right).
\end{align*}
Using the above inequalities with the case $k=0,1$ give (\ref{PDE-ODE ineq 1}) and the theorem is proved.
\end{proof}

\begin{proof}[Proof of theorem \ref{Backward uniqueness theorem}] In \cite{Kotschwar4}, the author proves a general backward uniqueness theorem for two sections $X$, $Y$ in two fixed vector bundles $\mathcal X$, $\mathcal Y$ on $\Sigma$ respectively. We remark that their proof goes through if one assume that $\mathcal X$, $\mathcal Y$ are both time dependent vector bundle with $\partial_t$ replaced by $D^{\mathcal X}_t, D_t^{\mathcal Y}$. In particular, to apply theorem 3 in \cite{Kotschwar4} to our situation, let $$\tau = T-t , \ \Lambda^{ij} = g^{ij}.$$
Note $\nabla \Lambda = 0$ and $b = \partial_\tau g, \nabla b, \partial_\tau \Lambda, R_\Sigma$ are all uniformly bounded, so is $[D_t, \nabla]$ since
$$[ D_t , \nabla] = \partial_t \Gamma + \bar R * H * F_* . $$
Thus theorem \ref{theorem PDE-ODE ineq} and theorem 3 in \cite{Kotschwar4} imply that $X = Y = 0$ on $[1/l, T]$. Thus $F = \widetilde F$ on $[1/l, T]$.
\end{proof}


\bibliographystyle{amsplain}

\end{document}